\documentclass[a4paper]{article}
\usepackage{epsfig,amsmath,amssymb,url,graphics,amsthm,verbatim,amsfonts,subfigure,psfrag}

\newtheorem{theorem}{Theorem}[section]

\newtheorem{proposition}[theorem]{Proposition}
\newtheorem{lemma}[theorem]{Lemma}

\newtheorem{remark}[theorem]{Remark}

\def\R{\mathbb{R}}

\def\N{\mathbb{N}}

\def\I{\infty}

\newcommand{\be}{\begin{equation}}
\newcommand{\ee}{\end{equation}}
\newcommand{\bea}{\begin{eqnarray}}
\newcommand{\eea}{\end{eqnarray}}
\newcommand{\beann}{\begin{eqnarray*}}
\newcommand{\eeann}{\end{eqnarray*}}
\newcommand{\benn}{\begin{equation*}}
\newcommand{\eenn}{\end{equation*}}

\newcommand{\xx}[1]{\,\text{ #1 }\,}
\newcommand{\Xx}[1]{\quad\text{ #1 }\,}
\newcommand{\XX}[1]{\quad\text{ #1 }\quad}

\def\ra{\rightarrow}
\def\I{\infty}

\newcommand{\cA}{{\mathcal A}}  
\newcommand{\cF}{{\mathcal F}}  
\newcommand{\cS}{{\mathcal S}}  

\renewcommand{\d}[1][x]{\,\operatorname{d}\!#1}

\newcommand{\dxi}{\d[\xi]}

\newcommand{\Riesz}{D_0^\alpha}  
\newcommand{\RieszFeller}{D_\theta^\alpha}  
\newcommand{\Green}{G_\theta^\alpha}  
\newcommand{\Fourier}{\mathcal{F}}
\newcommand{\FourierInv}{\mathcal{F}^{-1}}
\newcommand{\SchwartzTF}{\mathcal{S}} 
\newcommand{\integral}[3]{\int_{#1} \, {#2} \, \d[#3]\,}
\newcommand{\integrall}[4]{\int_{#1}^{#2} \, {#3} \, \d[#4]\,}
\newcommand{\abs}[1]{|#1|}

\newcommand{\difff}[3]{\tfrac{\partial^{#3} {#1}}{\partial {#2}^{#3}}}
\newcommand{\diff}[2]{\tfrac{\partial {#1}}{\partial {#2}}}

\newcommand{\norm}[1]{\| {#1} \|}

\newcommand{\um}{u_-}
\newcommand{\up}{u_+}
\newcommand{\upm}{u_\pm}

\DeclareMathOperator{\dom}{dom}

\DeclareMathOperator{\sgn}{sgn}

\begin{document}

\title{Analysis and numerics of traveling waves for asymmetric fractional reaction-diffusion equations}

\author{Franz Achleitner\thanks{Institute for Analysis and Scientific Computing, Vienna University of Technology, 1040 Vienna, Austria}, Christian Kuehn\thanks{Institute for Analysis and Scientific Computing, Vienna University of Technology, 1040 Vienna, Austria}}

\maketitle 
\begin{abstract}
We consider a scalar reaction-diffusion equation in one spatial dimension with bistable nonlinearity
 and a nonlocal space-fractional diffusion operator of Riesz-Feller type. 
We present our analytical results on the existence, uniqueness (up to translations) and stability of a traveling wave solution 
 connecting two stable homogeneous steady states. 
Moreover, we review numerical methods for the case of reaction-diffusion equations with fractional Laplacian
 and discuss possible extensions to our reaction-diffusion equations with Riesz-Feller operators.
In particular, we present a direct method using integral operator discretization in combination with projection boundary conditions to visualize our analytical results about traveling waves.
\end{abstract}



\section{Introduction.}
A scalar reaction-diffusion equation is a partial differential equation 
\begin{equation} \label{eq:RD:local}
   \diff{u}{t} = \difff{u}{x}{2} + f(u) \Xx{for} (x,t)\in\R\times (0,T]\,,
\end{equation}
where the spatial derivative models diffusion 
and (a nonlinear) function $f$ models reaction of some quantity $u=u(x,t)$ over time.  
The application and analysis of reaction-diffusion equations has a long
 history~\cite{Aronson+Weinberger:1974, Smoller:1994, Volpert+etal:1994}.

In the following,
 we consider equation~\eqref{eq:RD:local} with a \emph{bistable} nonlinear function $f\in C^1(\R)$ such that
 \begin{equation} \label{As:f:bistable} 
  \exists \um<a<\up \xx{in} \R\,: \quad
  f(u) \begin{cases}
  	    =0 & \xx{for} u\in \{\um\,, a\,, \up\}\,, \\
	    	<0 & \xx{for} u\in (\um,a)\,, \\
		   	>0 & \xx{for} u\in (a,\up)\,,
			 \end{cases}
 \end{equation}
 \[ f'(\um)<0\,, \quad f'(\up)<0\,. \] 

This kind of reaction-diffusion equation is known as
 Nagumo's equation to model propagation of signals~\cite{McKean:1970, Nagumo},
 as one-dimensional real Ginzburg-Landau equation (RGLE) to model long-wave amplitudes
  e.g. in case of convection in binary mixtures near the onset of instability~\cite{Newell+Whitehead:1969, Segel:1969}, 
 as well as Allen-Cahn equation to model phase transitions in solids~\cite{Allen+Cahn:1979}.

Following Allen and Cahn, 
 a stable stationary state - such as $\um$ and $\up$ - represents a phase of the system,
 whereas a traveling wave solution $u(x,t)=U(x-ct)$ with $\lim_{\xi\to\pm\infty} U(\xi)=\upm$ represents a phase transition. 
Each stationary state $u_*$ has an associated potential $F(u_*)=F(\um)+\integrall{\um}{u_*}{f(v)}{v}$.
One distinguishes between the balanced case,
 i.e. the stable states $\um$ and $\up$ have the same potential $F(\um)=F(\up)$,
 and the unbalanced case,
 where the stable state with lesser potential value $F(u)$ is called the metastable state.
Then a traveling wave solution $u(x,t)=U(x-ct)$ connecting the stable states $\um$ and $\up$
 will be stationary ($c=0$) in the balanced case 
 and moving in the direction of the metastable state in the unbalanced case.

In some applications it is important to include nonlocal effects.
For example, Bates et al.~\cite{Bates+etal:1997} proposed a non-local model 
\begin{equation} \label{eq:RD:Bates}
 \diff{u}{t} = J \ast u - u + f(u) \Xx{for} (x,t)\in\R\times (0,T]\,,
\end{equation}
for even, non-negative functions $J\in C^1(\R)$ with 
  \[ \integral{\R}{J(y)}{y}=1\,, \quad 
     \integral{\R}{\abs{y} J(y)}{y}<\infty\,, \quad
     J'\in L^1(\R)\,,
  \]
and bistable functions~$f$.
The assumptions on $J$ ensure that the problem exhibits
 a maximum principle and a variational formulation.
The existence of traveling wave solutions $u(x,t)=U(x-ct)$ is concluded
 from a homotopy of~\eqref{eq:RD:Bates} to a classical reaction-diffusion model~\eqref{eq:RD:local}.
Moreover the traveling wave again will move depending on the balance of the potential values of the stable states.
In contrast,
 the asymptotic stability is established only for stationary traveling wave solutions, i.e. in the balanced case,
 where an additional variational structure is available. 

Chen established a unified approach~\cite{Chen:1997} to prove the existence, uniqueness
 and asymptotic stability with exponential decay of traveling wave solutions
 for the previous reaction-diffusion equations and many more examples from the literature.
He considers general nonlinear nonlocal evolution equations in the form
 \[ \diff{u}{t} (x,t) = \mathcal{A}[u(\cdot,t)](x) \Xx{for} (x,t)\in\R\times (0,T]\,, \]
 where the nonlinear operator $\mathcal{A}$ is assumed to 
 \begin{enumerate}
  \item be independent of $t$;
  \item generate a $L^\infty$ semigroup;
  \item be translational invariant, i.e. $\cA$ satisfies for all $u\in\dom\cA$ the identity
   \[ \mathcal{A}[u(\cdot+h)](x) = \mathcal{A}[u(\cdot)](x+h)\quad \forall x\,,h \in\R\,. \]
   Consequently, there exists a function $f:\R\to\R$
   which is defined by $\cA[\alpha \mathbf 1]=f(\alpha)\mathbf 1$
	 for $\alpha\in\R$ and the constant function $\mathbf 1:\R\to\R$, $x\to 1$.
  This function $f$ is assumed to be bistable~\eqref{As:f:bistable};
  \item satisfy a comparison principle
   \begin{quote}
    If $\diff{u}{t}\geq\mathcal{A}[u]$, $\diff{v}{t}\leq\mathcal{A}[v]$ and $u(\cdot ,0)\gneq v(\cdot ,0)$,
    then $u(\cdot ,t)>v(\cdot ,t)$ for all $t>0$.
   \end{quote}
 \end{enumerate}
Chen's approach relies on the comparison principle
 and the construction of sub- and supersolutions for any given traveling wave solution.
Importantly, the method does not depend on the balance of the potential. 

At the same time, Zanette~\cite{Zanette:1997} proposed a model
 \begin{equation} \label{eq:RD:Zanette}
   \diff{u}{t} = \Riesz u + f(u) \Xx{for} (x,t)\in\R\times (0,T]\,,
 \end{equation}
 with a fractional Laplacian~$\Riesz$ for some $\alpha\in (0,2)$ 
 and an explicit bistable function $f$.
This model exhibits monotone traveling wave solutions having an explicit integral representation,
 hence the asymptotic behavior of front tails and the front width can be studied directly.
Subsequently, the reaction-diffusion equation~\eqref{eq:RD:Zanette}
 with fractional Laplacian and general bistable function~$f$
 has been studied in the literature~\cite{Zanette:1997, Nec+etal:2008, Volpert+etal:2010, Cabre+Sire:2010, Cabre+Sire:2011, Palatucci+etal:2013, Gui:2012, Chmaj:2013}. 

Engler~\cite{Engler:2010} was one of the first to consider the scalar partial integro-differential equations 
\begin{equation} \label{RD}
  \diff{u}{t} = \RieszFeller u + f(u) \Xx{for} (x,t)\in\R\times (0,T]\,,
\end{equation}
where $u=u(x,t)$, $f\in C^1(\R)$ is a (bistable) nonlinear function,
and $\RieszFeller$ is a Riesz-Feller operator with $1<\alpha\leq 2$ and $\abs{\theta}\leq \min\{\alpha,2-\alpha\}$.
A Riesz-Feller operator~$\RieszFeller$ of order $\alpha$ and skewness $\theta$ can be defined as a Fourier multiplier operator,
 see also the exposition of Mainardi, Luchko and Pagnini~\cite{MainardiLuchkoPagnini}.
Starting from the fundamental solution of $\diff{u}{t} = \RieszFeller u$, 
 Engler constructs traveling wave solutions for some appropriate bistable function~$f$.
Assuming the existence of traveling wave solutions for general functions~$f$,
 Engler studies the finiteness of the wave speed.
The existence, uniqueness (up to translations), and stability of traveling wave solutions for general bistable functions is left open.

\subsection{Main analytical result.}
\label{ssec:main}

Our main result is summarized in the following theorem. 
\begin{theorem}[\cite{AchleitnerKuehn1}]
\label{thm:main}
Suppose $1<\alpha\leq 2$, $\abs{\theta} \leq \min\{\alpha,2-\alpha\}$ and $f\in C^\infty(\R)$ satisfies~\eqref{As:f:bistable}.
Then equation~\eqref{RD} admits a traveling wave solution $u(x,t)=U(x-ct)$ satisfying
\begin{equation} \label{As:TWS}
  \lim_{\xi\to\pm\infty} U(\xi)=\upm \Xx{and} U'(\xi)>0 \Xx{for all} \xi\in\R\,.
\end{equation}
In addition, a traveling wave solution of~\eqref{RD} is unique up to translations.
Furthermore, traveling wave solutions are globally asymptotically stable in the sense that
 there exists a positive constant $\kappa$ such that 
 if $u(x,t)$ is a solution of~\eqref{RD} with initial datum~$u_0\in C_b(\R)$ satisfying $0\leq u_0\leq 1$ and
 \begin{equation} \label{As:stability}
  \liminf_{x\to\infty} u_0(x) > a\,, \qquad \limsup_{x\to-\infty} u_0(x) < a\,,
 \end{equation}
 then, for some constants $\xi$ and $K$ depending on $u_0$,
 \[ \norm{u(\cdot,t)-U(\cdot- ct + \xi)}_{L^\infty(\R)} \leq K e^{-\kappa t} \qquad \forall t\geq 0\,. \] 
\end{theorem} 

\subsection{Discussion.}
\label{ssec:discussion}
To our knowledge,
 we established the first result~\cite{AchleitnerKuehn1}
 on existence, uniqueness (up to translations) and stability of traveling wave solutions of~\eqref{RD}
 with Riesz-Feller operators~$\RieszFeller$ for $1<\alpha< 2$ and $|\theta| \leq \min\{\alpha,2-\alpha\}$
 and bistable functions~$f$ satisfying~\eqref{As:f:bistable}. The technical details of the proof are 
contained in~\cite{AchleitnerKuehn1}, whereas in this paper we give a concise overview of the proof strategy
and visualize the results also numerically. 

To prove Theorem~\ref{thm:main}, 
 we follow - up to some modifications - the approach of Chen~\cite{Chen:1997}.
His approach relies on the comparison principle
 and the construction of sub- and supersolutions for any given traveling wave solution.
It allows to cover all bistable functions~$f$ satisfying~\eqref{As:f:bistable} regardless of the balance of the potential
 and all Riesz-Feller operators~$\RieszFeller$ for $1<\alpha< 2$ regardless of $|\theta| \leq \min\{\alpha,2-\alpha\}$.

Next, we quickly review different methods to study the traveling wave problem
 of a reaction-diffusion equation.
In case of a classical reaction-local diffusion equation~\eqref{eq:RD:local},
 the existence of traveling wave solutions can be studied via phase-plane analysis~\cite{Aronson+Weinberger:1974, Fife+McLeod:1977}.
This method has no obvious generalization to our traveling wave problem for~\eqref{RD},
 since its traveling wave equation is an integro-differential equation.

The variational approach has been focused - so far - on symmetric diffusion operators such as fractional Laplacians 
 and on balanced potentials,
 hence covering only stationary traveling waves~\cite{Cabre+SolaMorales:2005, Cabre+Sire:2010, Cabre+Sire:2011, Palatucci+etal:2013}. The homotopy to a simpler traveling wave problem has been used to prove 
 the existence of traveling wave solutions in case of~\eqref{eq:RD:Bates},
 and~\eqref{eq:RD:Zanette} with unbalanced potential~\cite{Gui:2012}.

Chmaj~\cite{Chmaj:2013} also considers the traveling wave problem for~\eqref{eq:RD:Zanette} with general bistable functions~$f$.
He approximates a given fractional Laplacian by a family of operators $J_\epsilon \ast u - (\int J_\epsilon)u$
 such that $\lim_{\epsilon\to 0} J_\epsilon \ast u - (\int J_\epsilon) u = \Riesz u$ in an appropriate sense.
This allows him to obtain a traveling wave solution of~\eqref{eq:RD:Zanette} with general bistable function~$f$
 as the limit of the traveling wave solutions~$u_\epsilon$ of~\eqref{eq:RD:Bates} associated to $(J_\epsilon)_{\epsilon\geq 0}$. It might be possible to modify Chmaj's approach
 to study also our reaction-diffusion equation~\eqref{RD} with Riesz-Feller operators.
This would give an alternative existence proof of a traveling wave solutions.

However, Chen's approach allows to
 establish uniqueness (up to translations) and stability of traveling wave solutions as well. 
It remains an open problem to extend Chen's approach, if this is possible,
 to the general case of Riesz-Feller operators with $0<\alpha\leq 1$ and $\abs{\theta} \leq \min\{\alpha,2-\alpha\}$.

\subsection{Outline.}
Our article is structured as follows.
In Section~\ref{sec:TWP},
 we give a non-technical review of our analytical results in a companion article~\cite{AchleitnerKuehn1}.
We introduce the Riesz-Feller operators as Fourier multiplier operators on Schwartz functions,
 and extend the Riesz-Feller operators in form of singular integrals to functions in $C^2_b(\R)$.
The Riesz-Feller operators $\RieszFeller$ generate a convolution semigroup 
 which we deduce from the theory of L\'{e}vy processes.

Then we present the analysis of the the Cauchy problem for~\eqref{RD}
 with initial datum $u_0\in C_b(\R)$ such that $0\leq u_0\leq 1$.
The proof follows a standard approach, 
 to consider the Cauchy problem in its mild formulation
 and to prove the existence of a mild solution.
The Cauchy problem generates a nonlinear semigroup  
 which allows to prove uniform $C^k_b$ estimates via a bootstrap argument
 and to conclude that mild solutions are also classical solutions.

A comparison principle is essential to prove our result on the existence, uniqueness and stability of traveling wave solutions 
 and to allow for a larger class of admissible functions $f$ in the result for the Cauchy problem.

Finally, we consider the traveling wave problem for~\eqref{RD}.
In~\cite{AchleitnerKuehn1} we consider a general approach by Chen~\cite{Chen:1997}.
There we study his necessary assumptions and notice that some estimates are not of the required form.
However Chen's approach can be extended, which we prove in~\cite[Appendices~A--C]{AchleitnerKuehn1}. 
We sketch the proof of Theorem~\ref{thm:main} in Section~\ref{sec:TWP},
 and refer to~\cite[Subsection 4.2]{AchleitnerKuehn1} for more details.

In Section~\ref{sec:numerics},
 we review numerical methods for reaction-diffusion equations with fractional Laplacian
 and discuss the (im-)possibility of extensions to our reaction-diffusion equations with Riesz-Feller operators.
Then we present a direct method using integral operator discretization based on quadrature in combination with projection boundary conditions. Furthermore, we visualize the analytical results from Section~\ref{sec:TWP} and outline several challenges for the numerical analysis of asymmetric Riesz-Feller operators.

\section{Traveling wave solutions.}
\label{sec:TWP}
\label{sec:CP}
A Riesz-Feller operator of order $\alpha$ and skewness $\theta$ can be defined as a Fourier multiplier operator,
\begin{equation} \label{eq:RF:Fourier}
 \Fourier [\RieszFeller f](\xi) = \psi^\alpha_\theta (\xi) \Fourier [f] (\xi) \,,\qquad  \xi\in\R \,,
\end{equation}
with symbol
\begin{equation} \label{eq:RF:symbol}
  \psi^\alpha_\theta(\xi) = -|\xi|^\alpha \exp\left[i(\sgn(\xi)) \theta \tfrac{\pi}{2}\right] \,,
\end{equation}
for some $0<\alpha\leq 2$ and $|\theta| \leq \min\{\alpha, 2-\alpha\}$.
The symbol $\psi^\alpha_\theta (\xi)$ is the logarithm of the characteristic function
 of a L\'{e}vy strictly stable probability density with index of stability~$\alpha$ and asymmetry parameter~$\theta$
 according to Feller's parameterization \cite{Feller2, GorenfloMainardi}.

\begin{remark}
We follow the convention in probability theory
 and define the Fourier transform of $f$ in the Schwartz space $\cS(\R)$ as 
\begin{align*}
  \cF [f](\xi) &:= \integral{\R}{ e^{+i \xi x} f(x) }{x} \,, \qquad \xi\in\R \,,
\intertext{and the inverse Fourier transform as}
  \cF^{-1} [f](x) &:= \tfrac{1}{2\pi} \integral{\R}{ e^{-i \xi x} f(\xi) }{\xi} \,, \qquad x\in\R \,.
\end{align*}
Moreover, $\Fourier$ and $\FourierInv$ will denote also their respective extensions to $L^2(\R)$.
\end{remark}

To analyze the Cauchy problem for the reaction diffusion equation~\eqref{RD}
 we need to investigate the linear space-fractional diffusion equation
\begin{equation} \label{eq:linearRF}
 \diff{u}{t}(x,t) = \RieszFeller [u(\cdot ,t)](x) \Xx{for} (x,t)\in\R\times (0,\I)\,,
\end{equation}
$0<\alpha\leq 2$ and $\abs{\theta} \leq \min\{\alpha,2-\alpha\}$.
A formal Fourier transform of the associated Cauchy problem yields
\[ 
\diff{}{t} \Fourier[u](\xi,t) = \psi^\alpha_\theta(\xi) \Fourier[u](\xi,t)\,, 
  \qquad \Fourier[u](\xi,0) = \Fourier[u_0](\xi)\,, 
\] 
 which has a solution $\Fourier[u](\xi,t)=e^{t\psi_\theta^\alpha(\xi)}\Fourier[u_0](\xi)$. 
Hence, a formal solution of the Cauchy problem is given by 
\begin{equation} \label{SG}
 u(x,t) = (\Green(\cdot ,t) \ast u_0)(x)
\end{equation}
with kernel (or Green's function) $\Green(x,t): = \mathcal{F}^{-1} \left[\exp( t \psi^\alpha_\theta(\cdot))\right] (x)$.

Due to Theorem~\cite[Theorem 14.19]{Sato:1999}, 
 the function $e^{t\psi_\theta^\alpha(\xi)}$ is the characteristic function 
 of a random variable with L\'{e}vy strictly $\alpha$-stable distribution.
Thus $\Green$ is the scaled probability measure of a L\'{e}vy strictly $\alpha$-stable distribution.
In case of $(\alpha,\theta)\in\{(0,0),(1,1),(1,-1)\}$, 
 the probability measure $\Green$ is a delta distribution 
 \[ G^0_0(x,t) = \delta_x\,, \qquad G^1_1(x,t) = \delta_{x+t}\,, \qquad G^1_{-1}(x,t) = \delta_{x-t} \]
 and called trivial~\cite[Definition 13.6]{Sato:1999}.
In all other (non-trivial) cases,
 the probability measure $\Green$ is absolutely continuous with respect to the Lebesgue measure
 and has a continuous probability density~\cite[Proposition 28.1]{Sato:1999}, which we will denote again by $\Green$.
For every infinitely divisible distribution $\mu$ on $\R^d$, such as $\Green$,
 there exists an associated L\'evy process $(X_t)_{t\geq 0}$.
In particular, every L\'evy process exhibits an associated strongly continuous semigroup on $C_0(\R^d)$,
 see also~\cite[Theorem 31.5]{Sato:1999}.

The infinitesimal generator of our L\'evy process has the following representation,
 which allows to extend the Riesz-Feller operator to $C^2_b(\R)$-functions. 
\begin{theorem} \label{thm:RieszFeller:extension}
If $0<\alpha<1$ or $1<\alpha<2$ and $|\theta|\leq \min\{\alpha,2-\alpha\}$,
 then for all $f\in\SchwartzTF(\R)$ and $x\in\R$ 
\begin{multline} \label{eq:RieszFeller1}
\RieszFeller f(x) = \tfrac{c_1 - c_2}{1-\alpha} f'(x) +
  c_1 \integrall{0}{\infty}{ \tfrac{f(x+\xi)-f(x)-f'(x)\,\xi {\bf 1}_{(-1,1)}(\xi)}{\xi^{1+\alpha}} }{\xi} \\
  + c_2 \integrall{0}{\infty}{ \tfrac{f(x-\xi)-f(x)+f'(x)\,\xi {\bf 1}_{(-1,1)}(\xi)}{\xi^{1+\alpha}} }{\xi}
\end{multline}
where ${\bf 1}_{(-1,1)}(\cdot)$ is an indicator function 
and some constants $c_1, c_2 \geq 0$ with $c_1+c_2 >0$.
\end{theorem}
\begin{proof}
The result follows from~\cite[Theorem 31.7]{Sato:1999}
 see also~\cite[Theorem~2.4]{AchleitnerKuehn1}.
\end{proof}

In the analysis of the traveling wave problem, 
 we are mostly interested in the evolution of initial data in $C_b$.
Therefore, it is important to notice the following proposition.
\begin{proposition}[{\cite[Corollary 2.10]{AchleitnerKuehn1}}] \label{prop:Cb:semigroup}
For $1<\alpha< 2$ and $\abs{\theta} \leq \min\{\alpha,2-\alpha\}$,
 the Riesz-Feller operator~$\RieszFeller$ generates a convolution semigroup 
 $S_t: C_b(\R) \to C_b(\R)$, $u_0 \mapsto S_t u_0 = \Green(\cdot ,t)\ast u_0$,
 with kernel $\Green(x,t)$.
Moreover, the convolution semigroup with $u(x,t):=S_t u_0$ satisfies
\begin{enumerate}
\item $u\in C^\infty(\R\times(t_0,\infty))$ for all $t_0>0$;
\item $\frac{\partial u}{\partial t} = \RieszFeller u$ for all $(x,t)\in\R\times (t_0,\I)$ and any $t_0>0$;
\item If $u_0\in C_b(\R)$ then $u\in C_b(\R\times[0,T])$ for any $T>0$.
\end{enumerate}
\end{proposition}
This result states that
 Riesz-Feller operators $\RieszFeller$ for $0<\alpha\leq 2$ and $\abs{\theta}\leq \min\{\alpha,2-\alpha\}$
 generate conservative $C_b$-Feller semigroups. 
This can be deduced from a criterion on the symbol of Fourier multiplier operators in~\cite{Schilling:1998}.

It is important to notice that $S_t: C_b(\R) \to C_b(\R)$ is not a strongly continuous semigroup.
Thus the $C^2_b$-extension of $\RieszFeller$ are not the infinitesimal generators
 of the $C_b$-extension of the strongly continuous semigroup $(S_t)_{t\geq 0}$ on $C_0(\R)$
 in the usual sense.   

\subsection{Cauchy problem.}
We consider the Cauchy problem
\begin{equation} \label{CP:RD}
\begin{cases}
  \diff{u}{t} = \RieszFeller u + f(u) 	& \xx{for} (x,t)\in\R\times (0,\infty) \,, \\
  u(x,0) = u_0(x) 			& \xx{for} x\in\R \,,
\end{cases}
\end{equation}
for $1<\alpha\leq 2$, $|\theta| \leq \min\{\alpha, 2-\alpha\}$ and $f\in C^\infty(\R)$ satisfying~\eqref{As:f:bistable}.
We follow a standard approach, 
 and consider the Cauchy problem in its mild formulation
 to prove the existence of a mild solution.
The Cauchy problem generates a nonlinear semigroup  
 which allows to prove uniform $C^k_b$ estimates via a bootstrap argument
 and to conclude that mild solutions are also classical solutions.
 \begin{theorem}[{\cite[Theorem 3.3]{AchleitnerKuehn1}}] \label{thm:CP}
 Suppose $1<\alpha\leq 2$, $|\theta| \leq \min\{\alpha,2-\alpha\}$ and $f\in C^\infty(\R)$ satisfies~\eqref{As:f:bistable}.
 The Cauchy problem~\eqref{RD} with initial condition $u(\cdot ,0)=u_0\in C_b(\R)$ and $0\leq u_0\leq 1$
  has a solution $u(x,t)$ in the following sense: for all $T>0$
  \begin{enumerate} 
   \item 
     $u\in C_b(\R\times(0,T))$ and 
     $u\in C^\infty_b(\R\times(t_0,T))$ for all $t_0\in(0,T)$;
   \item 
     $u$ satisfies~\eqref{RD} on $\R\times(0,T)$;
   \item 
     If $u_0\in C_b(\R)$ then $u(\cdot,t)\to u_0$ uniformly as $t\to 0$;
   \item \label{As:DI:xx}
 	  $0\leq u(x,t) \leq 1$ for all $(x,t)\in\R\times(0,\infty)$;
   \item $\forall k\in\N$ $\forall t_0>0$ $\exists C>0$ such that $\norm{u(\cdot ,t)}_{C^k_b(\R)}\leq C$ $\forall 0<t_0<t$. 
  \end{enumerate}
 \end{theorem}

The following comparison principle is essential to prove our result on the existence, uniqueness and stability of traveling wave solutions 
 and to allow for a larger class of admissible functions $f$ in the result for the Cauchy problem.
\begin{lemma}[{\cite[Lemma 3.4]{AchleitnerKuehn1}}] \label{lem:comparison:Cb}
Assume $1<\alpha\leq 2$, $\abs{\theta} \leq \min\{\alpha,2-\alpha\}$, $T>0$
 and $u,v \in C_b(\R\times [0,T])\cap C^2_b(\R\times (t_0,T])$ for all $t_0\in(0,T)$ such that 
 \[
   \diff{u}{t} \leq \RieszFeller u + f(u) \XX{and} \diff{v}{t} \geq \RieszFeller v + f(v) \Xx{in} \R\times (0,T]\,. 
 \]
\begin{enumerate} 
\item If $v(\cdot ,0)\geq u(\cdot ,0)$ then $v(x,t)\geq u(x,t)$ for all $(x,t)\in\R\times (0,T]$.
\item If $v(\cdot ,0) \gneqq u(\cdot ,0)$ then $v(x,t) > u(x,t)$ for all $(x,t)\in\R\times (0,T]$.
\item Moreover,
 there exists a positive continuous function 
 \[ \eta:[0,\infty)\times(0,\infty)\to(0,\infty)\,, \quad (m,t)\mapsto\eta(m,t)\,, \]
 such that if $v(\cdot ,0)\geq u(\cdot ,0)$ then for all $(x,t)\in\R\times(0,T)$
 \[ v(x,t)-u(x,t) \geq \eta(\abs{x},t) \integrall{0}{1}{ v(y,0)-u(y,0) }{y} \,. \]
\end{enumerate}
\end{lemma}

\begin{proof}[Sketch of the proof of Theorem~\ref{thm:main}]
We present here a sketch of the proof of Theorem~\ref{thm:main}
 and refer to our article~\cite{AchleitnerKuehn1} for more details. 
To prove existence of traveling wave solutions satisfying~\eqref{As:TWS},
 we consider the Cauchy problem for~\eqref{RD} with some smooth initial datum $u_0\in C_b(\R)$ satisfying~\eqref{As:TWS}.
Due to Theorem~\ref{thm:CP} there exists a classical solution $u(x,t)$.
We consider a diverging sequence $\{t_j\}_{j\in\N}$ such that $\lim_{j\to\infty} t_j = \infty$
 and the associated sequence $\{ u(\cdot,t_j)\}_{j\in\N}$ in $C_b(\R)$.
Then, due to Arzela-Ascoli Theorem, 
 there exists a subsequence and a limiting function $\tilde u$ such that $\lim_{k\to\infty} u(\cdot,t_{j_k}) = \tilde u(\cdot)$.
The final and most important step is to verify that $\tilde u$ is a traveling wave solution of~\eqref{RD} satisfying~\eqref{As:TWS}.

To prove uniqueness (up to translations) of a traveling wave solution,
 sub- and super-solutions of~\eqref{RD} are constructed from any given traveling wave solution.
Assuming the existence of two traveling wave solutions,
 one traveling wave solution is bounded from below and from above 
 by suitable sub- and super-solutions associated to the other traveling wave solution, respectively.
The comparison principle in Lemma~\ref{lem:comparison:Cb} allows to show 
 that one traveling wave solution is a translated version of the other traveling wave solution.

To prove stability of a traveling wave solution,
 considering the Cauchy problem for~\eqref{RD} with initial datum $u_0$ satisfying~\eqref{As:stability},
 then the associated solution $v$ can be bounded from below and from above 
 by suitable sub- and super-solutions associated to the traveling wave solution, respectively.
The comparison principle and the evolution of sub- and super-solutions show
 that these bounds on the solution~$v$ get tighter 
 and allow to prove the exponential convergence to (a translated version of) the traveling wave solution. 

For more details see the proof of~\cite[Theorem 4.6]{AchleitnerKuehn1}.
\end{proof}

\section{Numerical methods.}
\label{sec:numerics}



In this section, we illustrate our results from Theorem~\ref{thm:main} and discuss numerical methods 
for \eqref{RD}. The case $\theta=0$ yields the fractional Laplacian 
$D^\alpha_0=-(-\Delta)^{\alpha/2}$ which has been discussed frequently from a numerical 
perspective in the literature. Hence, there is a notational convention to write \eqref{RD} 
for $\theta=0$ as
\benn
\frac{\partial u}{\partial t}+(-\Delta)^{\alpha/2}= f(u)\qquad \text{or} \qquad
\frac{\partial u}{\partial t}=-(-\Delta)^{\alpha/2}+ f(u).
\eenn
However, we shall adhere to the convention $D^\alpha_0$ as introduced previously.
First, we review some of the available numerical schemes for this
case. We restrict the computational domain from $x\in\R$ to $x\in[-b,b]=:\Omega$
for some (sufficiently large) $b>0$ and with Neumann or Dirichlet boundary conditions.
A numerical comparison of various methods for the case $D^\alpha_0$ has already been
carried out in \cite{SternEffenbergerFichtnerSchaefer,YangLiuTurner} so we shall focus our 
small survey in Sections \ref{ssec:Spectral}-\ref{ssec:mtransfer} on the difficulties in the 
numerical generalization from $\theta=0$ to $\theta\neq 0$ for space-fractional equations. 
Furthermore, we only cover spatial grid bases schemes and do not discuss stochastic particle 
methods.

The main novel results are our direct method using integral operator discretization in
combination with projection boundary conditions in Section \ref{ssec:ours} and the numerical 
results in Section \ref{ssec:numres} for \eqref{RD}. 

\subsection{Spectral methods.}
\label{ssec:Spectral}

One idea is to generalize \textit{spectral methods} to the fractional Laplacian 
case \cite{BuenoOrovioKayBurrage}. Let $\lambda_j$ denote the Laplacian eigenvalues
and $\phi_j$ the corresponding eigenfunctions for 
$D^2_0 \phi_l=\lambda_l\phi_l$ with $l\in \N_0=\N\cup\{0\}$. 
Consider $L^2(\Omega)$ then we may write $v\in L^2(\Omega)$ as a series expansion
\be
\label{eq:useries}
v=\sum_{l=0}^\I \hat{v}_l\phi_l,\qquad \hat{v}_l:=\langle v,\phi_l \rangle 
\ee
where $\langle \cdot,\cdot\rangle$ denotes the $L^2(\Omega)$ inner product. Fix some $\alpha$ 
with $1<\alpha\leq 2$ and consider
\be
H^{\alpha/2}(\Omega):=\left\{v\in L^2(\Omega):\sum_{l=0}^\I |\hat{v}_l|^2 
|\lambda_l|^{\alpha/2}<\I\right\}.
\ee
The spectral decomposition of the fractional
Laplacian implies \cite{BarriosColoradodePabloSanchez} that $-(-\lambda_l)^{\alpha/2}$ are 
eigenvalues with eigenfunctions $\phi_l$ for $D^\alpha_0$ and for any 
$v\in H^{\alpha/2}(\Omega)$ we have
\be
\label{eq:useries1}
D^\alpha_0v=-\sum_{l=0}^\I (-\lambda_l)^{\alpha/2}\hat{v}_l\phi_l.
\ee
As a remark, we note that all the minus signs on the right-hand side in \eqref{eq:useries1} 
disappear if we would write $(-\Delta)^{\alpha/2}u$ on the left-hand side instead and would let
$\lambda_l$ denote the eigenvalues of the negative Laplacian. It is suggested in 
\cite{BuenoOrovioKayBurrage} to apply a backward Euler-type time discretization on a mesh 
\be
\label{eq:temp_grid}
0=t_0<t_1<\cdots<t_m<t_{m+1}<\cdots<T
\ee
for \eqref{RD} where we set $t_{m+1}-t_m=:(\delta t)_m$. Denote by $u^{m}:=u(x,t_m)$ 
the solution at time $t_m$. For the time step $t_m$ to $t_{m+1}$ one may consider the 
semi-implicit backward Euler scheme
\be
\label{eq:Burrage1}
\frac{u^{m+1}-u^m}{(\delta t)_m}=D^\alpha_0 u^{m+1}+f(u^{m}).
\ee
Making the Fourier spectral ansatz
\benn
u(x,t)=\sum_{l=0}^\I \hat{u}_l(t)\phi_l(x)\approx \sum_{l=0}^L \hat{u}_l(t)\phi_l(x) 
\eenn
in \eqref{eq:Burrage1}, using the orthogonality of the basis functions $\phi_l$ and 
employing \eqref{eq:useries1} leads to the numerical method
\be
\label{eq:Burrage3}
\hat{u}_l^{m+1}=\frac{1}{1+(-\lambda_l)^{\alpha/2}(\delta t)_m}\left(\hat{u}_l^{m}+
(\delta t)_m \hat{f}_l(u^{m})\right)
\ee
where $\hat{f}_l$ is the $l$-th Fourier coefficient of $f$. In particular, the $L+1$ 
Fourier modes in \eqref{eq:Burrage3} are decoupled and relatively easy to solve for.
Further implementation details of \eqref{eq:Burrage3} can be found in 
\cite[Code 4,~p.10]{BuenoOrovioKayBurrage}. However, the generalization of
\eqref{eq:Burrage3} from the fractional Laplacian case $D^\alpha_0$ to the 
asymmetric case $D^\alpha_\theta$ with $\theta\neq 0$ is not straightforward. 
In fact, in the asymmetric case one generically obtains complex eigenvalues
and a continuous spectrum \cite{AlSaqabiBoyadjievLuchko}. This means that 
\eqref{eq:useries1} is no longer valid for $\theta\neq 0$. For another approach 
using transform/Fourier-type techniques we refer to \cite{SaxenaMathaiHaubold}. 

\subsection{Finite difference methods.}
\label{ssec:FDM}

A second possible approach to solve \eqref{RD} is to use a 
\textit{finite difference method} (FDM) \cite{MeerschaertTadjeran} 
combined with the Gr\"unwald-Letnikov representation of the space 
fractional derivative. Let us consider a spatial discretization of 
$\Omega=[-b,b]$ as follows
\be
\label{eq:space_grid}
-b=x_1<x_2<\cdots<x_N=b.
\ee
We still use the temporal discretization \eqref{eq:temp_grid}.
For $D^\alpha_0$ the Gr\"unwald-Letnikov representation of $D^\alpha_0$ is 
given by
\bea
(D^\alpha_0u)(x,t)&=&\lim_{N\ra \I}\frac{1}{h^\alpha_+}\sum_{r=0}^N
\frac{\Gamma(r-\alpha)}{\Gamma(-\alpha)\Gamma(r+1)}u(x-rh_+,t)\nonumber\\
&&+\lim_{N\ra \I}\frac{1}{h^\alpha_-}\sum_{r=0}^N
\frac{\Gamma(r-\alpha)}{\Gamma(-\alpha)\Gamma(r+1)}u(x+rh_+,t),
\eea
where $h_+=(x+b)/N$ and $h_-=(b-x)/N$. Let us assume for simplicity that
the spatial grid is equidistant and let $h:=2b/N$. Furthermore, we let 
\benn
u^{m}_n:=u(x_n,t_m).
\eenn
Then one possible finite-difference discretization of \eqref{RD} is 
given by \cite{MeerschaertTadjeran,SternEffenbergerFichtnerSchaefer}
\benn
\frac{u^{m+1}_n-u^{m}_n}{(\delta t)_m}=\frac{1}{h^\alpha}
\left[\sum_{r=0}^{n+1}g_r u^{m+1}_{n-r+1} + \sum_{r=0}^{N-n+1}g_r u^{m+1}_{n+r-1}\right]
+f(u^{m+1}_n),
\eenn
with $g_r:=\frac{\Gamma(r-\alpha)}{\Gamma(-\alpha)\Gamma(r+1)}$.
For a similar approach using the Gr\"unwald-Letnikov representation to obtain 
finite-difference schemes we also refer to 
\cite{ChenLiuTurnerAnh,GorenfloAbdelRehim,LiuAnhTurner,LiuZhuangAnhTurnerBurrage,
MeerschaertTadjeran1,SchererKallaTangHuang,TadjeranMeerschaertScheffler,ZhuangLiuAnhTurner}.
For even more details on finite-difference methods for space-fractional diffusion 
equations consider \cite{LanglandsHenry,ShenLiu,Yuste}. 
In some sense, our scheme in Section \ref{ssec:ours} has an analogous starting 
point. However, instead of the Gr\"unwald-Letnikov representation we use the integral
representation formula which we also employed in the existence-uniqueness-stability 
proof of Theorem~\ref{thm:main}; see also Section \ref{ssec:ours}.

\subsection{Finite element methods.}
\label{ssec:FEM}

Another quite natural possibility is to follow the classical 
\textit{finite element method} (FEM) variational 
approach. We follow \cite{ErvinHeuerRoop,JinLazarovPasciakZhou} 
in our exposition for the case $\theta=0$.
Let $X:=H^{\alpha/2}_0(\Omega)$ denote the usual fractional Sobolev space 
obtained as a closure of $C_0^\I(\Omega)$ in $H^{\alpha/2}(\Omega)$
and define
\be
\label{eq:form}
A(v,w):=-\langle D^{\alpha/2}_0v,D^{\alpha/2}_0w\rangle\,,
\ee
where representation~\eqref{eq:useries1} is used.
Then one may check that $A$ is coercive and continuous.
Consider the space $X_h$ of piecewise linear continuous functions in $X$ with compact support given by
\benn
X_h:=\{v\in C_0(\Omega):\text{$v$ is linear over $[x_n,x_{n+1}]$, $n=1,2,\ldots,N-1$}\}.
\eenn
Then we may define a discrete operator $A_h:X_h\ra X_h$ associated to $A$ via
\benn
\langle A_h v_h,w_h\rangle = A(v_h,w_h)\qquad \forall v_h,w_h\in X_h.
\eenn
A semi-discrete Galerkin FEM scheme for 
\eqref{RD} is to find $u_h=u_h(t)\in X_h$ such that
\be
\label{eq:Galerkin}
\left\langle \frac{\partial u_h}{\partial t}(t),v_h\right\rangle 
=\langle A_h u_h(t),v_h\rangle +\langle f(u_h(t)),v_h\rangle \qquad \forall v_h\in X\,,
\ee
and projected initial condition $\langle u_h(0),v_h\rangle=\langle u(0),v_h\rangle$.
Choosing a basis $\{\varphi_1,\varphi_2,\ldots,\varphi_N\}$ of $X_h$ we may
write
\benn
u_h(x,t)=\sum_{n=1}^N u_n(t)\varphi_n(x).
\eenn
One defines the usual mass matrix $M\in\R^{N\times N}$ and stiffness matrix 
$A\in\R^{N\times N}$ with entries
\be
\label{eq:mass_stiff_mat}
M_{nm}=\langle\varphi_n,\varphi_m \rangle,\qquad A_{nm}=A(\varphi_n,\varphi_m),\qquad m,n\in\{1,2,\ldots,N\}. 
\ee
This converts \eqref{eq:Galerkin} into the ODEs 
\be
\label{eq:FEM1}
M\frac{dU}{dt}=AU+f(U)
\ee
where $U=(u_1,\ldots,u_N)^T$ and 
$f(U)=(\langle f(u_h),\varphi_1\rangle,\ldots,\langle f(u_h),\varphi_N\rangle)^T$.
Then one may use a time-stepping scheme directly. For example, a backward Euler 
semi-implicit discretization yields
for $U^m:=U(t_m)$ the method
\be
\label{eq:FEM2}
(\text{Id}-(\delta t)_mM^{-1}A)U^{m+1}=U^{m}+(\delta t)_mM^{-1}f(U^{m}).
\ee
These considerations show that we can, at least formally, just follow the
classical FEM theory to derive numerical methods for equations involving $D^\alpha_0$. However, for 
the fractional Laplacian $D^\alpha_0$ the matrix entries for $A$ defined
in \eqref{eq:mass_stiff_mat} are not as easy to compute as for $D^2_0$. FEM techniques
also seem to generalize formally to the asymmetric case as coercivity and 
continuity hold for classes of fractional operators more general than 
$D^\alpha_0$ \cite[p.574-575]{ErvinHeuerRoop}. However, we are again
faced with the practical problem of computing (an approximation of) $A(\varphi_n,\varphi_m)$.
This observation is one reason which motivates the method presented in
the next section. For more details on FEM for space fractional equations we
refer to \cite{ErvinRoop,ErvinRoop1,FixRoof,Roop}.

\subsection{Matrix-transfer techniques.}
\label{ssec:mtransfer}

The symmetry of $D^2_0$ and the view of fractional powers $D^\alpha_0$
can be employed in conjunction with FEM or FDM discretizations for 
\eqref{RD}. Again, we consider the case $\theta=0$ following 
\cite{IlicLiuTurnerAnh,IlicLiuTurnerAnh1,BurrageHaleKay}. Let $A_\Delta\in\R^{N\times N}$ 
be the usual FEM stiffness matrix and $M_\Delta$ be the FEM mass matrix for $D^2_0$. 
One natural idea is to use a fractional power of the matrix 
$B_\Delta:=M_\Delta^{-1}A_\Delta$ in a numerical 
scheme to represent the fractional Laplacian. Suppose we can compute 
$(B_\Delta)^\alpha$ then a backward semi-implicit Euler-type time discretization, 
similar to \eqref{eq:FEM2}, leads to
\be
\label{eq:Burrage4}
(\textnormal{Id} -(\delta t)_m (B_\Delta)^\alpha)U^{m+1}=U^{m}+M^{-1}_\Delta f(U^{m}).
\ee  
To solve \eqref{eq:Burrage4} one has to also compute the function 
\be
\label{eq:matrix_func_special}
q(z)=\frac{1}{1-(\delta t)_m z^\alpha}
\ee
efficiently for matrices, which has been discussed in \cite{BurrageHaleKay}. 
However, we still have to define $B^\alpha_\Delta$. Standard theory implies that $A_\Delta,M_\Delta$ 
are real, symmetric matrices \cite{ErnGuermond}. Furthermore, $A_\Delta$ is non-negative definite
and $M_\Delta$ is positive definite. A direct calculation
shows that
\begin{align*}
(M_\Delta)^{1/2} B_\Delta (M_\Delta)^{-1/2}
 &= (M_\Delta)^{-1/2}  A_\Delta (M_\Delta)^{-1/2}.
\end{align*}
Therefore, $B_\Delta$ is similar to a real, symmetric matrix with well-defined
point spectrum $\sigma(B_\Delta)\subset \R$ and eigenvalues 
$\xi_1\leq \xi_2\leq \cdots\leq \cdots \xi_N$. Then it is 
very natural to define a matrix function $q(Z)$, including 
\eqref{eq:matrix_func_special} as a special case, by 
\benn
q(Z)=Qq(\Xi)Q^{-1},
\eenn
where $\Xi$ is a diagonal matrix with $\Xi_{nn}=\xi_n$, 
$Q$ consists of the eigenvectors associated to the eigenvalues $\xi_n$
and $[q(\Xi)]_{nn}=q(\xi_n)$. This yields a well-defined 
fractional power $(B_\Delta)^\alpha$ when applied to $q(z)=z^\alpha$ and 
can then also be applied to define \eqref{eq:matrix_func_special}.
Unfortunately, the matrix transfer technique does not generalize immediately
to the case $\theta\neq 0$ as the spectrum $\sigma(D^\alpha_\theta)$
for $\theta\neq 0$ is generically continuous with complex eigenvalues as 
already discussed in Section \ref{ssec:Spectral}. For more on the matrix
transfer technique we refer to \cite{YangTurnerLiu,YangTurnerLiuIlic}.

\subsection{Integral representation, quadrature and projection boundary conditions.}
\label{ssec:ours}

Sections \ref{ssec:Spectral}-\ref{ssec:mtransfer} explain why, to the best of our 
knowledge, there seem to be very few (if any) detailed numerical studies of the 
asymmetric case $\theta\neq0$ for the nonlinear Allen-Cahn/Nagumo-type Riesz-Feller 
reaction-diffusion equation \eqref{RD}. 

Here we present an easy-to-implement method to study \eqref{RD} 
numerically with a focus on the dynamics of traveling waves.
Our approach is to use the integral representation of Riesz-Feller 
operators to view \eqref{RD} as an integro-differential equation.
For $\alpha\in(1,2)$ the representation formula is given by 
\cite{AchleitnerKuehn1}
\bea
\label{eq:rep}
(D^\alpha_\theta u)(x,t)&=&c_1\int_0^\I \frac{u(x+\xi,t)-u(x,t)-\xi\frac{\partial u}{\partial x}(x,t)}{\xi^{1+\alpha}}\dxi\nonumber\\
&&+c_2\int_0^\I \frac{u(x-\xi,t)-u(x,t)+\xi\frac{\partial u}{\partial x}(x,t)}{\xi^{1+\alpha}}\dxi
\eea
where the constants $c_{1,2}$ are given in \cite{MainardiLuchkoPagnini} as
\benn
c_1=\frac{\Gamma(1+\alpha)\sin\left((\alpha+\theta)\frac\pi2\right)}{\pi}\quad\text{and}
\quad c_2=\frac{\Gamma(1+\alpha)\sin\left((\alpha-\theta)\frac\pi2\right)}{\pi}.
\eenn
Note that there is also an integral representation for $\alpha\in(0,1)$ \cite{AchleitnerKuehn1}
for $x\in\R$. Furthermore, there is an analogous integral representation formula for fractional Laplacians in $\R^d$ 
in \cite{DroniouImbert}. Hence, starting from a representation like \eqref{eq:rep} is not really
a restriction, even for higher-dimensional cases. Furthermore, a similar strategy has also been applied 
successfully in a similar to context to other nonlocal operator equations \cite{AlibaudAzeradIsebe} 
involving traveling waves. If we write 
\beann
g_1(\xi,x,t)&:=&\frac{u(x+\xi,t)-u(x,t)-\xi\frac{\partial u}{\partial x}(x,t)}{\xi^{1+\alpha}},\\
g_2(\xi,x,t)&:=&\frac{u(x-\xi,t)-u(x,t)+\xi\frac{\partial u}{\partial x}(x,t)}{\xi^{1+\alpha}},
\eeann
then we can simply re-write \eqref{RD} as an integro-differential equation
\be
\label{eq:PIDE}
\frac{\partial u}{\partial t}(x,t)=c_1\int_0^\I g_1(\xi,x,t)\dxi
+c_2\int_0^\I g_2(\xi,x,t)\dxi+f(u(x,t)).
\ee
For simplicity, we shall just introduce our method for a uniform spatial 
mesh \eqref{eq:space_grid}, {i.e.} we have
\be
\label{eq:space_grid_uniform}
-b=x_1<x_2<\cdots<x_N=b\qquad \text{with}\quad x_{n+1}-x_n=\frac{2b}{N-1}=:h,
\ee
where we assume that $N\geq 3$ is odd so that $x_{(N+1)/2}=0$. Furthermore, 
we use another spatial mesh to approximate the integral operators \eqref{eq:rep} 
over a finite domain obtained as a sub-mesh from \eqref{eq:space_grid_uniform} as
follows
\be
\xi_1=x_{\frac{N+1}{2}+1},\quad \xi_2=x_{\frac{N+1}{2}+2},\quad \ldots,\quad\xi_{M+1}=x_N,
\ee
which has $M$ subintervals $[\xi_m,\xi_{m+1}]$. We may easily relate $M$ to the
number of points $N$ in our original mesh by $M=(N-1)/2$. For $b,N$
sufficiently large we may just use a quadrature rule to approximate \eqref{eq:rep};
we remark that the possibility to use quadrature techniques for time-fractional 
Caputo-derivative fractional equations has already been noticed in 
\cite{Agrawal,KumarAgrawal}. Very recently (in fact, during the preparation of this work), 
Huang and Oberman \cite{HuangOberman} proposed a quadrature-scheme based upon a 
singular integral presentation of the symmetric case $D^\alpha_0$. We use the 
regularized, fully asymmetric representation \eqref{eq:rep} and obtain for the trapezoidal 
rule, with $\rho\in\{1,2\}$, that
\begin{multline*}
c_\rho\int_0^\I g_\rho(\xi,x,t)\dxi \approx c_\rho\int_h^b g_\rho(\xi,x,t)\dxi \\
 \approx \frac{c_\rho(b-h)}{2M}\left[g_\rho(\xi_1,x,t)+g_\rho(\xi_{M+1},x,t)+2\sum_{m=2}^Mg_\rho(\xi_m,x,t)\right].
\end{multline*}
Using this approximation in \eqref{eq:PIDE} yields a system of (formal) ODEs for $u_n(t)=u(x_n,t)$, 
 which can be written as 
\bea
\frac{\textnormal{d}u_n}{\textnormal{d}t}&=&\frac{b-h}{2M}\left[c_1\left((g_1(u))_{n,1}+(g_1(u))_{n,M+1}\right)+2c_1\sum_{m=2}^M(g_1(u))_{n,m}
\right.\nonumber\\
&&\left.+c_2\left((g_2(u))_{n,1}+(g_2(u))_{n,M+1}\right)+2c_2\sum_{m=2}^M(g_2(u))_{n,m}\right]+f(u_n)\label{eq:MoL}
\eea
for $n\in\{1,2,\ldots,N\}$, where the terms involving $(g_\rho(u))_{n,m}$ are given by
\benn
(g_1(u))_{n,m}=\tfrac{u_{n+m}-u_n}{\xi_m^{1+\alpha}}-\tfrac{u_{n+1}-u_n}{\xi_m^{\alpha}h}, 
\quad (g_2(u))_{n,m}=\tfrac{u_{n-m}-u_n}{\xi_m^{1+\alpha}}+\tfrac{u_{n+1}-u_n}{\xi_m^{\alpha}h}.
\eenn
Of course, the system \eqref{eq:MoL} is, as yet, only a formal representation as it involves spatial 
mesh indices for $u$ which lie outside the range {i.e.}~$u_n=u(x_n,t)$ for $n\in\{1,2,\ldots,N\}$. There
is a choice of boundary conditions. However, instead of classical Neumann or Dirichlet conditions, we
want to compute traveling waves which satisfy
\benn
\lim_{x\ra-\I}u(x,t)=u_-,\qquad \lim_{x\ra+\I}u(x,t)=u_+
\eenn
for constants $u_\pm$. Hence, we adopt the following projection-type boundary conditions for the
numerical method
\be
\label{eq:bc}
u_n=\left\{\begin{array}{ll}u_N &\text{ if $n\geq N$,}\\ u_1 &\text{ if $n\leq 1$,}\end{array}\right.
\ee
Using \eqref{eq:bc}, we get a well-defined ODE system \eqref{eq:MoL} which can be solved using 
forward integration, {i.e.}~we adopt a method-of-lines approach; for more details on using 
projection boundary conditions to compute traveling waves in the classical FitzHugh-Nagumo equation 
we refer {e.g.} to \cite{Doedel_AUTO1997,GuckenheimerKuehn1,GuckenheimerKuehn3}.

Regarding our algorithm \eqref{eq:MoL}-\eqref{eq:bc} for waves of the Riesz-Feller
bistable equation, we emphasize that our approach is clearly non-optimal
from a numerical perspective. For example, there are straightforward generalizations 
to non-uniform meshes and higher-order schemes by using non-uniform-mesh higher-order 
quadrature methods. We leave these generalizations as future challenges. Here, we are 
primarily interested in developing a simple scheme for \eqref{RD} and to visualize 
some of the results from Theorem~\ref{thm:main}.

\subsection{Numerical results.}
\label{ssec:numres}

In this section, we briefly discuss some numerical simulations of~\eqref{RD}
 with $f(u)=u(1-u)(u-a)$ for some $a\in(0,1)$
 using our algorithm from Section \ref{ssec:ours}. 
Unless stated otherwise, we fix $\Omega=[-b,b]=[-30,30]$, $N=181$, spatial mesh points and  
always employ a standard stiff ODE solver to solve \eqref{eq:MoL}-\eqref{eq:bc} 
(more precisely, \texttt{ode15s} from MatLab \cite{ShampineReichelt}) for $t\in[0,T]$. 
Figure \ref{fig:fig01}(a) shows the initial condition  
\be
\label{eq:ic_Chen}
u_0(x)=u(x,0)=\left\{\begin{array}{ll}
0 &\text{if $x\in[-30,-2)$},\\ 
\frac14x+\frac12 &\text{if $x\in[-2,2]$},\\ 
1 &\text{if $x\in(2,30]$}.\\ 
\end{array}\right.
\ee
The initial condition \eqref{eq:ic_Chen} is important as it has been
used in the existence part of the proof of Theorem~\ref{thm:main} as discussed in \cite{AchleitnerKuehn1,Chen1}.
In particular, $u_0$ is shown to converge to a traveling wave.

\begin{figure}[htbp]
\psfrag{a}{(a)}
\psfrag{b}{(b)}
\psfrag{c}{(c)}
\psfrag{x}{$x$}
\psfrag{t}{$t$}
\psfrag{u}{$u$}
\psfrag{u0}{$u_0$}
\centering
\includegraphics[width=1\textwidth]{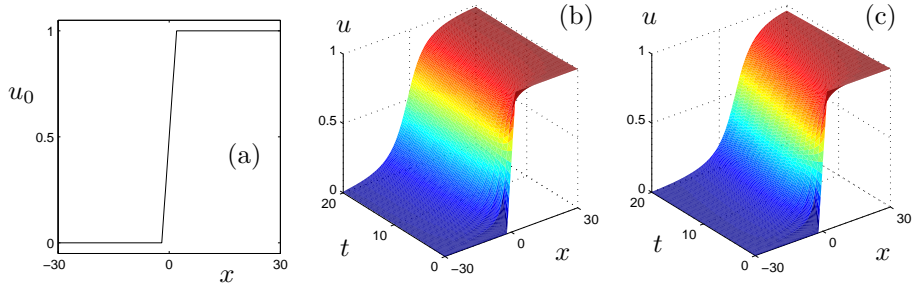}
\caption{Fixed parameter values are $\theta=0.1$, $\alpha=1.8$, $T=20$. 
 (a) Initial condition $u_0=u(x,0)$ given by \eqref{eq:ic_Chen}.
 (b) Simulation with $a=0.5$.
 (c) Simulation with $a=0.6$, the wave travels to the right.
}
\label{fig:fig01}
\end{figure} 

Figure \ref{fig:fig01}(b)-(c) show the fully asymmetric fractional case with $D^\alpha_\theta$ 
for $\alpha=1.8$ and $\theta=0.1$. 
In both cases we observe a rapid smoothing effect of the solution
 as predicted by the smoothing result in Theorem~\ref{thm:main}. 
Furthermore, in both cases, convergence to a traveling wave profile is observed,
 where moving the parameter $a$ changes the wave speed. Again, this is 
expected since the supremum-norm of the nonlinearity $f(u)=u(1-u)(u-a)$ does influence the 
wave speed.

\begin{figure}[htbp]
\psfrag{a}{(a)}
\psfrag{b}{(b)}
\psfrag{c}{(c)}
\psfrag{x}{$x$}
\psfrag{t}{$t$}
\psfrag{u}{$u$}
	\centering
		\includegraphics[width=1\textwidth]{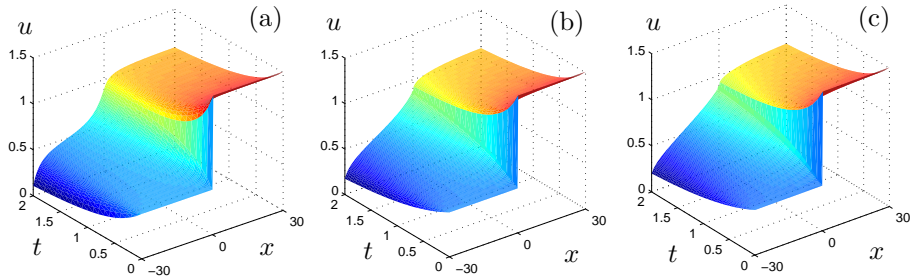}
	\caption{\label{fig:fig02}Fixed parameter values are $\theta=0.1$, $a=0.5$, $T=2$ with
	initial condition \eqref{eq:ic1}.
	(a) $\alpha=1.8$. (b) $\alpha=1.2$. (c) $\alpha=1.01$.}
\end{figure} 

As a second interesting part we are interested in discontinuous initial conditions bounded
away from the traveling wave, and even violating one of the stability assumptions ($0\leq u_0\leq 1$)
from Theorem~\ref{thm:main}. One example is
\be
\label{eq:ic1}
u_0(x)=u(x,0)=\left\{\begin{array}{ll}
0.49 &\text{if $x\in[-30,0]$},\\ 
1.51 &\text{if $x\in(0,30]$}.\\ 
\end{array}\right.
\ee
Furthermore, we vary the fractional exponent $\alpha$. Figure \ref{fig:fig02} shows the results.
Although the initial condition is not within the framework of the theoretical analysis, we still
observe extremely rapid convergence to a wave profile where the end-states move to $u(-b,t)=0$ and 
$u(b,t)=1$. Note however, that the convergence, as well
as the regularization effect, seems to be slower for smaller exponents $\alpha$.

\begin{figure}[htbp]
\psfrag{a}{(a)}
\psfrag{b}{(b)}
\psfrag{c}{(c)}
\psfrag{x}{$x$}
\psfrag{t}{$t$}
\psfrag{u}{$u$}
	\centering
		\includegraphics[width=1\textwidth]{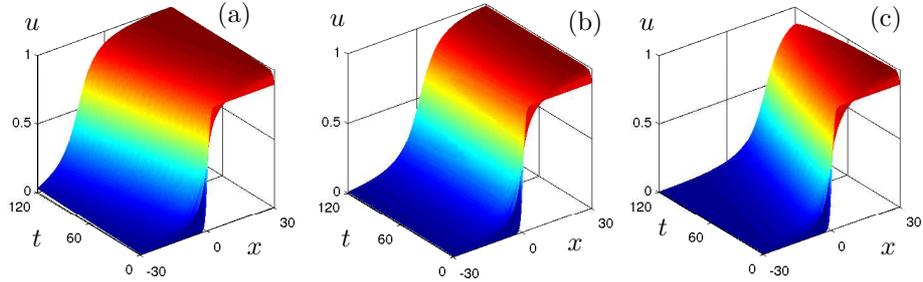}
	\caption{\label{fig:fig03}Fixed parameter values are $\alpha=1.5$, $a=0.5$, $T=120$ with
	initial condition \eqref{eq:ic_Chen}.
	(a) $\theta=0.2$; wave moves to the left. (b) $\theta=0.0$; standing wave. (c) $\theta=-0.2$; wave moves to the right.}
\end{figure} 

Another question is the effect of the asymmetry parameter $\theta$. Figure \ref{fig:fig03} shows
three different cases for $\theta=0.2,0,-0.2$. It is clearly visible that the wave speed is
directly affected. Within the time $t\in[0,T]$, the wave in Figure \ref{fig:fig02}(b) barely moves
while there is a drift to the right in Figure \ref{fig:fig02}(c) and to the left in
Figure \ref{fig:fig02}(a). Hence, we may conclude that the asymmetry parameter definitely has an 
effect on quantitative properties of traveling waves. Based on the relation to microscopic 
super-diffusion processes and previous studies for other nonlinearities \cite{MancinelliVergniVulpiani},
a quantitative change is expected.

\begin{figure}[htbp]
\psfrag{a}{(a)}
\psfrag{b}{(b)}
\psfrag{c}{(c)}
\psfrag{x}{$x$}
\psfrag{t}{$t$}
\psfrag{u}{$u$}
	\centering
		\includegraphics[width=1\textwidth]{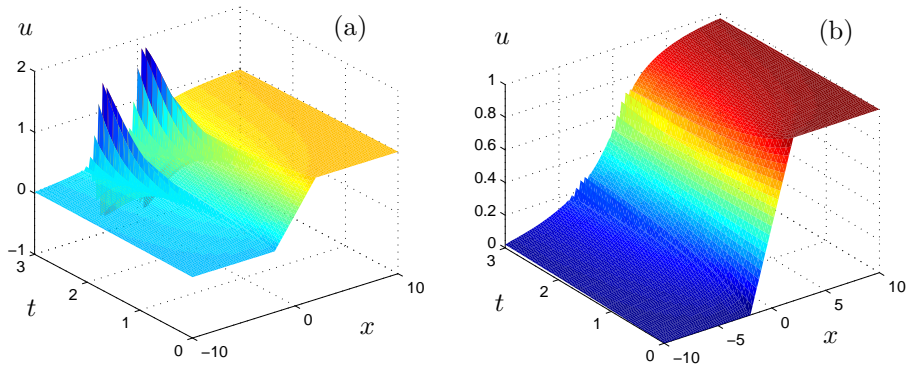}
	\caption{\label{fig:fig04}Fixed parameter values are $\alpha=1.5$, $a=0.5$, $T=3$, $\theta=0.4$ with
	$b=10$ and $L=101$ {i.e.}~on a coarser grid than in the previous figures.
	(a) Absolute tolerance for the ODE time stepper is $10^{-6}$. 
	(b) Absolute tolerance for the ODE time stepper is $10^{-9}$.}
\end{figure} 

As a last issue, we briefly discuss the influence of the asymmetry parameter on numerical
stability. Figure \ref{fig:fig04} shows simulations for the same parameter values $\alpha=1.5$,
$\theta=0.4$ where $\theta$ is chosen closer to the critical line $2-\alpha$ (see Section~\ref{sec:TWP}) than
before. The absolute error tolerance for the numerical time step is different in Figures 
\ref{fig:fig04}(a)-(b). Whereas we observe numerically induced oscillations in Figure \ref{fig:fig04}(a)
for a relatively low tolerance, the oscillations are suppressed for the more accurate computation 
in Figure \ref{fig:fig04}(b). We checked that the numerical solution poses no problem for the 
lower error tolerance when $\theta$ is lower as well, for example, $\theta=0.1$. This gives a
strong indication that the ODE problem may be stiff, respectively that the region of A-stability
shrinks when $\theta$ is changed. In particular, this leads to the conjecture that the asymmetric
case is not only more complicated with respect to the design and implementation of numerical 
algorithms but also with respect to numerical stability.

\subsection{Numerical analysis: some challenges.}

In this section, we would like to highlight some numerical challenges/conjectures which are 
relevant for future work:

\begin{enumerate}
 \item Provide a generalization of our scheme to higher-order quadrature rules and non-uniform 
 meshes, including convergence and error analysis. 
 \item Generalize the scheme to $2$- and $3$-dimensional cases. What about the computation of 
 coherent/localized structures for this case?
 \item Investigate the numerical stability properties of algorithms for asymmetric fractional evolution 
 equations regarding the $(\alpha,\theta)$-dependence.
 \item Compare various approaches to truncate the domain $\R$.
 What is the influence of boundary conditions for space-fractional equations?
 \item What about adaptive algorithms to resolve wave profiles? What is the influence of $\alpha$
 and $\theta$ on the adaptive mesh selection?
 \item Provide robust methods, including error estimates, to calculate the wave-speed and 
 far-field/tail behavior. 
 \item Which methods for fractional diffusion equations, derived by different approaches such 
 as FDM, FEM or quadrature, are equivalent? 
\end{enumerate}

\section*{Acknowledgements.}

CK would like to thank the Austrian Academy of Sciences ({\"{O}AW}) 
for support via an APART fellowship. CK also acknowledges the European Commission (EC/REA) for 
support by a Marie-Curie International Re-integration Grant.

\newpage 

\section{Bibliography.}

\baselineskip=0.9\normalbaselineskip

\bibliographystyle{CAIMbibstyle}

\bibliography{bibliography}

\begin{thebibliography}{10}

\bibitem{Aronson+Weinberger:1974}
D.~Aronson and H.~Weinberger, Nonlinear diffusion in population genetics,
  combustion, and nerve pulse propagation, in {\em Partial Differential
  Equations and Related Topics}, vol.~446 of {\em Lecture Notes in
  Mathematics}, pp.~5--49, Springer, 1974.

\bibitem{Smoller:1994}
J.~Smoller, {\em Shock waves and reaction-diffusion equations}.
\newblock Springer, 1994.

\bibitem{Volpert+etal:1994}
A.~I. Volpert, V.~A. Volpert, and V.~A. Volpert, {\em Traveling wave solutions
  of parabolic systems}, vol.~140 of {\em Translations of Mathematical
  Monographs}.
\newblock Providence, RI: American Mathematical Society, 1994.
\newblock Translated from the Russian manuscript by James F. Heyda.

\bibitem{McKean:1970}
H.~McKean, {Nagumo's equation}, {\em Advances in mathematics}, vol.~4,
  pp.~209--223, 1970.

\bibitem{Nagumo}
J.~Nagumo, S.~Arimoto, and S.~Yoshizawa, An active pulse transmission line
  simulating nerve axon, {\em Proc. IRE}, vol.~50, pp.~2061--2070, 1962.

\bibitem{Newell+Whitehead:1969}
A.~Newell and J.~Whitehead, Finite bandwidth, finite amplitude convection, {\em
  The Journal of Fluid Mechanics}, vol.~38, p.~279–303, 1969.

\bibitem{Segel:1969}
L.~Segel, Distant side-walls cause slow amplitude modulation of cellular
  convection, {\em The Journal of Fluid Mechanics}, vol.~38, p.~203–224,
  1969.

\bibitem{Allen+Cahn:1979}
S.~M. Allen and J.~W. Cahn, A microscopic theory for antiphase boundary motion
  and its application to antiphase domain coarsening, {\em Acta Metallurgica},
  vol.~27, no.~6, pp.~1085--1095, 1979.

\bibitem{Bates+etal:1997}
P.~W. Bates, P.~C. Fife, X.~Ren, and X.~Wang, {Traveling waves in a convolution
  model for phase transitions}, {\em Archive for Rational Mechanics and
  Analysis}, vol.~138, pp.~105--136, 1997.

\bibitem{Chen:1997}
X.~Chen, Existence, uniqueness, and asymptotic stability of travelling waves in
  nonlocal evolution equations, {\em Adv. Differential Equations}, vol.~2,
  pp.~125--160, 1997.

\bibitem{Zanette:1997}
D.~Zanette, {Wave fronts in bistable reactions with anomalous L\'{e}vy-flight
  diffusion}, {\em Physical Review E}, vol.~55, no.~1, pp.~1181--1184, 1997.

\bibitem{Nec+etal:2008}
Y.~Nec, A.~A. Nepomnyashchy, and A.~A. Golovin, Front-type solutions of
  fractional {A}llen-{C}ahn equation, {\em Phys. D}, vol.~237, no.~24,
  pp.~3237--3251, 2008.

\bibitem{Volpert+etal:2010}
V.~A. Volpert, Y.~Nec, and A.~A. Nepomnyashchy, {Exact solutions in front
  propagation problems with superdiffusion}, {\em Physica D: Nonlinear
  Phenomena}, vol.~239, pp.~134--144, Feb. 2010.

\bibitem{Cabre+Sire:2010}
X.~Cabr\'{e} and Y.~Sire, {Nonlinear equations for fractional Laplacians I:
  Regularity, maximum principles, and Hamiltonian estimates}, {\em arXiv
  preprint arXiv:1012.0867}, vol.~01, pp.~1--41, 2010.

\bibitem{Cabre+Sire:2011}
X.~Cabr\'{e} and Y.~Sire, {Nonlinear equations for fractional Laplacians II:
  existence, uniqueness, and qualitative properties of solutions}, {\em arXiv
  preprint arXiv:1111.0796}, vol.~01, 2011.

\bibitem{Palatucci+etal:2013}
G.~Palatucci, O.~Savin, and E.~Valdinoci, {Local and global minimizers for a
  variational energy involving a fractional norm}, {\em Annali di Matematica
  Pura ed Applicata}, vol.~192, pp.~673--718, Jan. 2013.

\bibitem{Gui:2012}
C.~Gui and M.~Zhao. workshop presentation
  \url{http://birs.ca/events/2012/5-day-workshops/12w5100/videos}, 2012.

\bibitem{Chmaj:2013}
A.~Chmaj, {Existence of traveling waves in the fractional bistable equation},
  {\em Archiv der Mathematik}, vol.~100, pp.~473--480, May 2013.

\bibitem{Engler:2010}
H.~Engler, On the speed of spread for fractional reaction-diffusion equations,
  {\em Int. J. Differ. Equ.}, pp.~Art. ID 315421, 16, 2010.

\bibitem{MainardiLuchkoPagnini}
F.~Mainardi, Y.~Luchko, and G.~Pagnini, The fundamental solution of the
  space-time fractional diffusion equation, {\em Fract. Calc. Appl. Anal.},
  vol.~4, no.~2, pp.~153--192, 2001.

\bibitem{AchleitnerKuehn1}
F.~Achleitner and C.~Kuehn, ``Traveling waves for a bistable equation with
  nonlocal-diffusion.'' http://arxiv.org/abs/1312.6304v1, 2013.

\bibitem{Fife+McLeod:1977}
P.~Fife and J.~McLeod, The approach of solutions nonlinear diffusion equations
  to travelling front solutions, {\em Arch. Rational Mech. Anal.}, vol.~65,
  pp.~335--361, 1977.

\bibitem{Cabre+SolaMorales:2005}
X.~Cabr\'{e} and J.~Sol\`{a}-Morales, {Layer solutions in a half-space for
  boundary reactions}, {\em Communications on Pure and Applied Mathematics},
  vol.~LVIII, pp.~1678--1732, 2005.

\bibitem{Feller2}
W.~Feller, {\em An Introduction to Probability Theory and its Applications},
  vol.~2.
\newblock Wiley, 2nd~ed., 1972.

\bibitem{GorenfloMainardi}
R.~Gorenflo and F.~Mainardi, Random walk models for space-fractional diffusion
  processes, {\em Fract. Calc. Appl. Anal.}, vol.~1, no.~2, pp.~167--191, 1998.

\bibitem{Sato:1999}
K.~Sato, {\em L\'evy processes and infinitely divisible distributions}, vol.~68
  of {\em Cambridge Studies in Advanced Mathematics}.
\newblock Cambridge: Cambridge University Press, 1999.
\newblock Translated from the 1990 Japanese original, Revised by the author.

\bibitem{Schilling:1998}
R.~L. Schilling, Conservativeness and extensions of {F}eller semigroups, {\em
  Positivity}, vol.~2, no.~3, pp.~239--256, 1998.

\bibitem{SternEffenbergerFichtnerSchaefer}
R.~Stern, F.~Effenberger, H.~Fichtner, and T.~Sch{\"a}fer, The space-fractional
  diffusion-advection equation: analytical solutions and critical assessment of
  numerical solutions, {\em arXiv:1309.4263}, pp.~1--20, 2013.

\bibitem{YangLiuTurner}
Q.~Yang, F.~Liu, and I.~Turner, Numerical methods for fractional partial
  differential equations with {Riesz} space fractional derivatives, {\em Appl.
  Math. Model.}, vol.~34, no.~1, pp.~200--218, 2010.

\bibitem{BuenoOrovioKayBurrage}
A.~Bueno-Orovio, D.~Kay, and K.~Burrage, Fourier spectral methods for
  fractional-in-space reaction-diffusion equations, {\em preprint}, pp.~1--19,
  2012.

\bibitem{BarriosColoradodePabloSanchez}
B.~Barrios, E.~Colorado, A.~de~Pablo, and U.~S\'{a}nchez, On some critical
  problems for the fractional {Laplacian} operator, {\em J. Differential
  Equat.}, vol.~252, no.~11, pp.~6133--6162, 2012.

\bibitem{AlSaqabiBoyadjievLuchko}
B.~Al-Saqabi, L.~Boyadjiev, and Y.~Luchko, Comments on employing the
  {Riesz-Feller} derivative in the {Schr\"odinger} equation, {\em Eur. Phys. J.
  Special Topics}, vol.~222, pp.~1779--1794, 2013.

\bibitem{SaxenaMathaiHaubold}
R.~Saxena, A.~Mathai, and H.~Haubold, Computational solutions of distributed
  order reaction-diffusion systems associated with {Riemann-Liouville}
  derivatives, {\em arXiv:1211.0063v1}, pp.~1--12, 2012.

\bibitem{MeerschaertTadjeran}
M.~Meerschaert and C.~Tadjeran, Finite difference approximations for two-sided
  space-fractional partial differential equations, {\em Appl. Numer. Math.},
  vol.~56, no.~1, pp.~80--90, 2006.

\bibitem{ChenLiuTurnerAnh}
C.~Chen, F.~Liu, I.~Turner, and V.~Anh, A {Fourier} method for the fractional
  diffusion equation describing sub-diffusion, {\em J. Comp. Phys.}, vol.~227,
  no.~2, pp.~886--897, 2007.

\bibitem{GorenfloAbdelRehim}
R.~Gorenflo and E.~Abdel-Rehim, Convergence of the {Gr\"unwald-Letnikov} scheme
  for time-fractional diffusion, {\em J. Comp. Appl. Math.}, vol.~205, no.~2,
  pp.~871--881, 2007.

\bibitem{LiuAnhTurner}
F.~Liu, V.~Anh, and I.~Turner, Numerical solution of the space fractional
  fokker-planck equation, {\em J. Comput. Appl. Math.}, vol.~166, no.~1,
  pp.~209--219, 2004.

\bibitem{LiuZhuangAnhTurnerBurrage}
F.~Liu, P.~Zhuang, V.~Anh, I.~Turner, and K.~Burrage, Stability and convergence
  of the difference methods for the space-time fractional advection-diffusion
  equation, {\em Appl. Math. Comput.}, vol.~191, no.~1, pp.~12--20, 2007.

\bibitem{MeerschaertTadjeran1}
M.~Meerschaert and C.~Tadjeran, Finite difference approximations for fractional
  advection-dispersion flow equations, {\em J. Comp. Appl. Math.}, vol.~172,
  no.~1, pp.~65--77, 2004.

\bibitem{SchererKallaTangHuang}
R.~Scherer, S.~Kalla, Y.~Tang, and J.~Huang, The {Gr\"unwald-Letnikov} method
  for fractional differential equations, {\em Comput. Math. Appl.}, vol.~62,
  no.~3, pp.~902--917, 2011.

\bibitem{TadjeranMeerschaertScheffler}
C.~Tadjeran, M.~Meerschaert, and H.~Scheffler, A second-order accurate
  numerical approximation for the fractional diffusion equation, {\em J.
  Comput. Phys.}, vol.~213, no.~1, pp.~205--213, 2006.

\bibitem{ZhuangLiuAnhTurner}
P.~Zhuang, F.~Liu, V.~Anh, and I.~Turner, Numerical methods for the
  variable-order fractional advection-diffusion equation with a nonlinear
  source term, {\em SIAM J. Numer. Anal.}, vol.~47, no.~3, pp.~1760--1781,
  2009.

\bibitem{LanglandsHenry}
T.~Langlands and B.~Henry, The accuracy and stability of an implicit solution
  method for the fractional diffusion equation, {\em J. Comp. Phys.}, vol.~205,
  no.~2, pp.~719--736, 2005.

\bibitem{ShenLiu}
S.~Shen and F.~Liu, Error analysis of an explicit finite difference
  approximation for the space fractional diffusion equation with insulated
  ends, {\em ANZIAM J.}, vol.~46, pp.~C871--C887, 2005.

\bibitem{Yuste}
S.~Yuste, Weighted average finite difference methods for fractional diffusion
  equations, {\em J. Comp. Phys.}, vol.~216, no.~1, pp.~264--274, 2006.

\bibitem{ErvinHeuerRoop}
V.~Ervin, N.~Heuer, and J.~Roop, Numerical approximation of a time dependent
  nonlinear space-fractional diffusion equation, {\em SIAM J. Numer. Anal.},
  vol.~45, no.~2, pp.~572--591, 2007.

\bibitem{JinLazarovPasciakZhou}
B.~Jin, R.~Lazarov, J.~Pasciak, and Z.~Zhou, Error analysis of finite element
  methods for space-fractional parabolic equations, {\em arXiv:1310.0066v1},
  pp.~1--20, 2013.

\bibitem{ErvinRoop}
V.~Ervin and J.~Roop, Variational formulation for the stationary fractional
  advection dispersion equation, {\em Numer. Meth. PDE}, vol.~22, no.~3,
  pp.~558--576, 2006.

\bibitem{ErvinRoop1}
V.~Ervin and J.~Roop, Variational solution of fractional advection dispersion
  equations on bounded domains in {$\mathbb{R}^d$}, {\em Numer. Meth. PDE},
  vol.~23, no.~2, pp.~256--281, 2007.

\bibitem{FixRoof}
G.~Fix and J.~Roof, Least squares finite-element solution of a fractional order
  two-point boundary value problem, {\em Comput. Math. with Appl.}, vol.~48,
  no.~7, pp.~1017--1033, 2004.

\bibitem{Roop}
J.~Roop, Computational aspects of {FEM} approximation of fractional advection
  dispersion equations on bounded domains in {$\mathbb{R}^2$}, {\em J. Comp.
  Appl. Math.}, vol.~193, no.~1, pp.~243--268, 2006.

\bibitem{IlicLiuTurnerAnh}
M.~Ili{\'c}, F.~Liu, I.~Turner, and V.~Anh, Numerical approximation of a
  fractional-in-space diffusion equation {I}, {\em Frac. Calc. Appl. Anal.},
  vol.~8, no.~3, pp.~323--341, 2005.

\bibitem{IlicLiuTurnerAnh1}
M.~Ili{\'c}, F.~Liu, I.~Turner, and V.~Anh, Numerical approximation of a
  fractional-in-space diffusion equation {II} - with nonhomogeneous boundary
  conditions, {\em Frac. Calc. Appl. Anal.}, vol.~9, no.~4, pp.~333--349, 2006.

\bibitem{BurrageHaleKay}
K.~Burrage, N.~Hale, and D.~Kay, An efficient implicit {FEM} scheme for
  fractional-in-space reaction-diffusion equations, {\em SIAM J. Sci. Comput.},
  vol.~34, no.~4, pp.~A2145--A2172, 2012.

\bibitem{ErnGuermond}
A.~Ern and J.-L. Guermond, {\em Theory and Practice of Finite Elements}.
\newblock Springer, 2004.

\bibitem{YangTurnerLiu}
Q.~Yang, I.~Turner, and F.~Liu, Analytical and numerical solutions for the time
  and space-symmetric fractional diffusion equation, {\em ANZIAM J.}, vol.~50,
  pp.~C800--C814, 2009.

\bibitem{YangTurnerLiuIlic}
Q.~Yang, I.~Turner, F.~Liu, and M.~Ilic, Novel numerical methods for solving
  the time-space fractional diffusion equation in two dimensions, {\em SIAM J.
  Sci. Comput.}, vol.~33, no.~3, pp.~1159--1180, 2011.

\bibitem{DroniouImbert}
J.~Droniou and C.~Imbert, Fractal first-order partial differential equations,
  {\em Arch. Rational Mech. Anal.}, vol.~182, pp.~299--331, 2006.

\bibitem{AlibaudAzeradIsebe}
N.~Alibaud, P.~Azerad, and D.~Is{\`e}be, A non-monotone nonlocal conservation
  law for dune morphodynamics, {\em Differen. Integral Equat.}, vol.~23, no.~1,
  pp.~155--188, 2010.

\bibitem{Agrawal}
O.~Agrawal, A numerical scheme for initial compliance and creep response of a
  system, {\em Mech. Res. Commun.}, vol.~36, no.~4, pp.~444--451, 2009.

\bibitem{KumarAgrawal}
P.~Kumar and O.~Agrawal, An approximate method for numerical solution of
  fractional differential equations, {\em Signal Processing}, vol.~86, no.~10,
  pp.~2602--2610, 2006.

\bibitem{HuangOberman}
Y.~Huang and A.~Oberman, Numerical methods for the fractional {Laplacian} part
  {I}: a finite difference-quadrature approach, {\em arXiv:1311.7691v1},
  pp.~1--24, 2013.

\bibitem{Doedel_AUTO1997}
E.~Doedel, ``Auto 97: Continuation and bifurcation software for ordinary
  differential equations.'' http://indy.cs.concordia.ca/auto, 1997.

\bibitem{GuckenheimerKuehn1}
J.~Guckenheimer and C.~Kuehn, Homoclinic orbits of the fitzhugh-nagumo
  equation: The singular limit, {\em DCDS-S}, vol.~2, no.~4, pp.~851--872,
  2009.

\bibitem{GuckenheimerKuehn3}
J.~Guckenheimer and C.~Kuehn, Homoclinic orbits of the fitzhugh-nagumo
  equation: Bifurcations in the full system, {\em SIAM J. Appl. Dyn. Syst.},
  vol.~9, pp.~138--153, 2010.

\bibitem{ShampineReichelt}
L.~Shampine and M.~Reichelt, The {MatLab ODE} suite, {\em SIAM J. Sci.
  Comput.}, vol.~18, no.~1, pp.~1--22, 1997.

\bibitem{Chen1}
X.~Chen, Existence, uniqueness, and asymptotic stability of travelling waves in
  nonlocal evolution equations, {\em Adv. Differential Equations}, vol.~2,
  pp.~125--160, 1997.

\bibitem{MancinelliVergniVulpiani}
R.~Mancinelli, D.~Vergni, and A.~Vulpiani, Front propagation in reactive
  systems with anomalous diffusion, {\em Physica D}, vol.~185, no.~3,
  pp.~175--195, 2003.

\end{thebibliography}

\end{document}